\documentclass{article}
\title{Erratum: Limit theorems for Betti numbers of random simplicial complexes}
\author{Matthew Kahle and Elizabeth Meckes}
\date{April 24, 2014}
\usepackage{amsfonts,amssymb,amsthm,eucal,amsmath,bbold,verbatim,relsize}
\usepackage[margin=1.25in]{geometry}
\usepackage[usenames, dvipsnames]{color}

\usepackage{xy}
\input xy
\xyoption{all}
\allowdisplaybreaks
\usepackage{graphicx}
\usepackage{calc}
\newlength{\depthofsumsign}
\newcommand{\nsum}[1][1.4]{
    \mathop{%
        \raisebox
            {-#1\depthofsumsign+1\depthofsumsign}
            {\scalebox
                {#1}
                {$\displaystyle\sum$}%
            }
    }
}

\newtheorem{thm}{Theorem}[section]

\newtheorem{lemma}[thm]{Lemma}

\theoremstyle{definition}

\newcommand{\R}{\mathbb{R}}

\newcommand{\E}{\mathbb{E}}
\newcommand{\N}{\mathbb{N}}

\newcommand{\vol}{\mathop{\mathrm{vol}}}

\newcommand{\ds}{\displaystyle}

\newcommand{\X}{\mathcal{X}}
\newcommand{\Y}{\mathcal{Y}}
\renewcommand{\P}{\mathbb{P}}

\newcommand{\var}{\mathrm{Var}}
\newcommand{\cov}{\mathrm{Cov}}

\renewcommand{\S}{\mathcal{S}}
\newcommand{\p}{\mathcal{P}}
\renewcommand{\1}{\mathbb{1}}

\begin{document}

\maketitle

\begin{abstract}
We correct the proofs of the main theorems in our earlier paper ``Limit theorems for Betti numbers of random simplicial complexes.'' We are grateful to D.\ Yogeshwaran for pointing out the mistakes.
\end{abstract}

\section{The Erd\H{o}s-Renyi random clique complex}
In the paper \cite{KM}, we claimed a central limit theorem for the Betti number of an Erd\H{o}s-Renyi random simplicial complex (Theorem 2.4).  The proof given there contains an error, however, with minor modifications and an additional recent result, the proof goes through essentially as before.  For brevity, we refer to \cite{KM} for notation and descriptions of the models used.  Formally, we have the following modification of Theorem 2.4 of \cite{KM}.
\begin{thm}\label{ER-CLT}
Consider the Erd\H{o}s-Renyi clique complex $X(n,p)$; that is, take a random 1-skeleton on $n$ vertices, in which edges are present independently and with probability $p$, and let $X(n,p)$ be the maximal complex over this 1-skeleton.   Suppose that there is some $\delta > 0$ such that $p = \omega(n^{-1/k + \delta})$ and $p=o( n^{-1/(k+1) - \delta})$.
Then $$ \frac{
  \beta_k (X(n,p)) - \E[\beta_k(X(n,p))] }{ \sqrt{\var [\beta_k]}} \Rightarrow
\mathcal{N}(0,1).$$

\end{thm}
Note that the range of $p$ is slightly restricted relative to what was claimed earlier, when $\delta$ was taken to be 0.

\medskip

The mistake in the proof given in \cite{KM} of this result was the claim that, given the Morse inequalities
\[f_k-f_{k+1}-f_{k-1}\le\beta_k\le f_k, \]
and central limit theorems for the recentered, renormalized upper and lower bounds, a central limit theorem for $\beta_k$ itself followed; this is not true, however, because the difference in means of the upper and lower bounds is too large relative to the normalization to allow such a conclusion.  However, the proof of the central limit theorem for the lower bound is valid, and with minor modifications, one can use the same proof to obtain a central limit theorem for the quantity
\[\tilde{\beta}_k:=f_k-f_{k+1}-f_{k-1}+f_{k+2}+f_{k-2}-f_{k+3}-f_{k-3}+\cdots.\]
A consequence of the results in \cite{clique2} is that for $p$ in the given regime, a.a.s.\ all the Betti numbers except for $\beta_k$ are zero.  It then follows immediately from the Euler formula that $\beta_k=\tilde{\beta}_k$ a.a.s.

A central limit theorem for $\beta_k$ is then essentially immediate from a central limit theorem for $\tilde{\beta}_k$:
\begin{equation*}\begin{split}\P\left[\frac{\beta_k-\E \tilde{\beta}_k}{\sqrt{\var(\tilde{\beta}_k)}}\le t\right]&\le\P\left[\frac{\tilde{\beta}_k-\E \tilde{\beta}_k}{\sqrt{\var(\tilde{\beta}_k)}}\le t,\beta_k=\tilde{\beta_k}\right]+\P[\beta_k\neq\tilde{\beta}_k]\le\P\left[\frac{\tilde{\beta}_k-\E \tilde{\beta}_k}{\sqrt{\var(\tilde{\beta}_k)}}\le t\right]+\P[\beta_k\neq\tilde{\beta}_k];\end{split}\end{equation*}
the corresponding lower bound follows the same way.  Thus
\begin{equation}\label{not-quite}
\left|\P\left[\frac{\beta_k-\E \tilde{\beta}_k}{\sqrt{\var(\tilde{\beta}_k)}}\le t\right]-\Phi(t)\right|\le\left|\P\left[\frac{\tilde{\beta}_k-\E \tilde{\beta}_k}{\sqrt{\var(\tilde{\beta}_k)}}\le t\right]-\Phi(t)\right|+\P[\beta_k\neq\tilde{\beta}_k].\end{equation}
For $k$ fixed, the second quantity tends to zero as $n\to\infty$ and if $\tilde{\beta}_k$ satisfies a central limit theorem, then so does the first quantity, and we are done.

To move to the actual statement of Theorem \ref{ER-CLT}, one needs a slightly more refined version of the a.a.s. equality of $\beta$ and $\tilde{\beta}_k$.  In fact, the techniques in \cite{clique2} give that for $p$ in the given regime, $\beta_k=\tilde{\beta}_k$ with probability $1 - o \left( n^{-M} \right)$ for any constant $M > 0$. On the other hand, since a simplicial complex on $n$ vertices has $\beta_k \le f_k \le {n \choose {k+1}}$, in all cases we have that $\beta_k-\tilde{\beta}_k = O \left( n^{k+1} \right)$.
It is shown below that $\var(\tilde{\beta}_k)\sim n^{2k}p^{2\binom{k+1}{2}-1}$.  These estimates together are enough to obtain Theorem \ref{ER-CLT}, as follows.  First note that the above estimates imply that
\[\frac{|\E\beta_k-\E\tilde{\beta_k}|}{\sqrt{\var(\tilde{\beta}_k)}}=o\left(n^{1-M}p^{\frac{1}{2}-\binom{k+1}{2}}\right);\]
choosing $M\ge \frac{k+3}{2}$, we have that in the regime of $p$ specified in the theorem, $\frac{|\E\beta_k-\E\tilde{\beta_k}|}{\sqrt{\var(\tilde{\beta}_k)}}\to0$ as $n\to\infty$.
Now, 
\[\frac{\var(\beta_k)}{\var(\tilde{\beta}_k)}=1+\frac{\var(\beta_k-\tilde{\beta}_k)}{\var(\tilde{\beta}_k)}+2\frac{\cov(\tilde{\beta}_k,\beta_k-\tilde{\beta}_k)}{\var(\tilde{\beta}_k)}.\]
By the Cauchy-Schwarz inequality,
\(\frac{|\cov(\tilde{\beta}_k,\beta_k-\tilde{\beta}_k)|}{\var(\tilde{\beta}_k)}\le\sqrt{\frac{\var(\beta_k-\tilde{\beta}_k)}{\var(\tilde{\beta}_k)}},\)
and 
\[\frac{\var(\beta_k-\tilde{\beta}_k)}{\var(\tilde{\beta}_k)}\le\frac{\E|\beta_k-\tilde{\beta}_k|^2}{\var(\tilde{\beta}_k)}=o\left(n^{2-M}p^{1-2\binom{k+1}{2}}\right),\]
which tends to zero if $M\ge k+3$.  
We thus have that $\frac{\var(\beta_k)}{\var(\tilde{\beta}_k)}\to1$ as $n\to\infty$.  Finally, we get that for any $\delta,\delta',\epsilon>0$, for $n$ large enough, 
\begin{equation*}\begin{split}
\P\left[\frac{\beta_k-\E \beta_k}{\sqrt{\var(\beta_k)}}\le t\right]&=\P\left[\frac{\beta_k-\E \tilde{\beta}_k}{\sqrt{\var(\tilde{\beta}_k)}}\le t\sqrt{\frac{\var(\beta_k)}{\var(\tilde{\beta}_k)}}+\left(\frac{\E\beta_k-\E\tilde{\beta}_k}{\sqrt{\var(\tilde{\beta}_k)}}\right)\right]\\&\le\P\left[\frac{\beta_k-\E \tilde{\beta}_k}{\sqrt{\var(\tilde{\beta}_k)}}\le t(1+\delta)+\delta'\right]\\&\le\Phi(t(1+\delta)+\delta')+\epsilon.
\end{split}\end{equation*}
Letting $\delta,\delta',\epsilon$ tend to zero, we have 
\[\limsup_{n\to\infty}\P\left[\frac{\beta_k-\E \beta_k}{\sqrt{\var(\beta_k)}}\le t\right]\le \Phi(t).\]
The lower bound is proved the same way, and thus it suffices to prove a central limit theorem for $\tilde{\beta}_k$.  To do this, we follow essentially the same proof as the one given in \cite{KM} for $f_k-f_{k-1}-f_{k+1}$.  In particular, the result is an application of the following result of Barbour, Karo{\'n}ski, and Ruci{\'n}ski.
\begin{thm}Let $\{X_{\bf j}:{\bf j}=(j_1,\ldots,j_r)\in J\}$ be a dissociated set of random variables, such that
 $\E X_{\bf j}=0$ for all ${\bf j}$.  Let $W:=\sum_{{\bf j}\in J}X_{\bf j}$ and suppose that the $X_{\bf j}$ are normalized such that $\E W^2=1$.  Then 
\begin{equation}\label{BKR-thm}
d_1(W,Z)\le K\sum_{{\bf j}\in J}\sum_{{\bf k},
{\bf l}\in L_{\bf j}}\Big[\E|X_{\bf j}X_{\bf k}X_{\bf l}|+\E|X_{\bf j}
X_{\bf k}|\E|X_{\bf l}|\Big],
\end{equation}
where $Z$ is a standard normal random variable and $L_{\bf j}$ is the dependency neighborhood of ${\bf j}$.
\end{thm}

Write 
\[\tilde{\beta}_k=\sum_{A\subseteq \mathcal{V}}(-1)^{|A|+k+1}\xi_A,\]
where $\mathcal{V}$ is the collection of $n$ vertices over which our complex is built and $\xi_A$ is the indicator that $A$ spans a complete graph in the random complex $X(n,p)$; that is, that all $\binom{|A|}{2}$ potential edges between vertices in $A$ are present.  Let $\sigma^2:=\var\left(\tilde{\beta}_k\right)$, and consider the random variable 
\[W:=\frac{\tilde{\beta}_k-\E\tilde{\beta}_k}{\sqrt{\var(\tilde{\beta}_k)}}=\frac{1}{\sigma}\sum_{A\subseteq\mathcal{V}}(-1)^{|A|+k+1}\big(\xi_A-\E\xi_A\big);\]
that is,
\[X_A:=\frac{(-1)^{|A|+k+1}}{\sigma}\big(\xi_A-\E\xi_A\big).\]
It is not hard to see that for any subsets $A,B,C$,
\[\E|X_AX_BX_C|+\E|X_AX_B|\E|X_C|\le \frac{16}{\sigma^3}\E\big[\xi_A\xi_B\xi_C\big],\]
and it thus suffices to estimate
\[\frac{16}{\sigma^3}\sum_{A\subseteq\mathcal{V}}\sum_{B,C\in L_A}\E\big[\xi_A\xi_B\xi_C\big],\]
where for $A\subseteq\mathcal{V}$, $L_A$ is the collection of subsets of $\mathcal{V}$ sharing at least two vertices with $A$ (so that they have at least one potential edge in common).  Decomposing by the sizes of $A,B,C$ and the sizes of their intersections, we have that
\begin{equation*}\begin{split}
\frac{16}{\sigma^3}&\sum_{A\subseteq\mathcal{V}}\sum_{B,C\in L_A}\E\big[\xi_A\xi_B\xi_C\big]\\&=\frac{16}{\sigma^3}\nsum[2]_{\substack{\ell_A,\ell_B,\ell_C\ge 2\\r_{A,B}\ge 2\\r_{A\setminus B,C}\ge 0\\r_{A\cap B,C}\ge (2-r_{A\setminus B,C})_+\\r_{B\setminus A,C}\ge 0}}\mbox{\Large{C}}\,p^{\binom{\ell_A}{2}+\binom{\ell_B}{2}+\binom{\ell_C}{2}-\binom{r_{A,B}}{2}-\binom{r_{A\setminus B,C}+r_{A\cap B,C}}{2}-\binom{r_{B\setminus A,C}+r_{A\cap B,C}}{2}+\binom{r_{A\cap B,C}}{2}},
\end{split}\end{equation*}
where the upper limits all depend only on $k$, and the combinatorial coefficient {\Large{C}} is given by 
\begin{equation*}
\begin{split}\mbox{\Large{C}}&=\binom{n}{\ell_A}\binom{n-\ell_A}{\ell_B-r_{A,B}}\binom{\ell_A}{r_{A,B}}\binom{\ell_B-r_{A,B}}{r_{B\setminus A,C}} \\& \qquad\qquad\times \binom{r_{A,B}}{r_{A\cap B,C}}\binom{\ell_A-r_{A,B}}{r_{A\setminus B,C}}\binom{n-\ell_A-\ell_B+r_{A,B}}{\ell_C-r_{A\setminus B,C}-r_{B\setminus A,C}-r_{A\cap B,C}}\\&\le c_k n^{\ell_A+\ell_B+\ell_C-r_{A,B}-r_{A\setminus B,C}-r_{B\setminus A,C}-r_{A\cap B,C}},\end{split}
\end{equation*}
for a constant $c_k$ depending only on $k$.
If we fix $\ell_A,\ell_B,r_{A,B}$ and ignore for the moment those factors that depend only on these parameters, we are left with sums over $\ell_C$, etc. of terms of size
\begin{equation}\label{noAB}\frac{1}{\sigma^3}n^{\ell_C-r_{A\setminus B,C}-r_{B\setminus A,C}-r_{A\cap B,C}}p^{\binom{\ell_C}{2}-\binom{r_{A\setminus B,C}+r_{A\cap B,C}}{2}-\binom{r_{B\setminus A,C}+r_{A\cap B,C}}{2}+\binom{r_{A\cap B,C}}{2}}.\end{equation}
Now, if $\ell_C$ is increased by one and the new element of $C$ is also in, say $A\setminus B$, then the power of $n$ in the expression above does not change, but the power of $p$ does; the ratio of the new term to the old is
\[p^{\binom{\ell_C+1}{2}-\binom{\ell_C}{2}-\binom{r_{A\setminus B,C}+1 +r_{A\cap B,C}}{2}+\binom{r_{A\setminus B,C}+r_{A\cap B,C}}{2}}=p^{\ell_c-r_{A\setminus B,C}-r_{A\cap B,C}}.\]
Similarly, if $\ell_C$ is increased by one and the new element of $C$ is also in $A\cap B$, then the ratio of the new term to the old is
\[p^{\ell_C-r_{A\setminus B,C}-r_{B\setminus A,C}-r_{A\cap B,C}}.\]
Since in both cases the power on $p$ is nonnegative, adding a new vertex to $C$ which is already in $A\cup B$ can only make the summand smaller.  On the other hand, if $\ell_C$ is increased by 1, and the new vertex is not in $A$ or $B$, then the ratio of the new term to the old is
\(np^{\ell_C}.\)  In the regime that we consider, this tends to infinity for (the old) $\ell_C\le k$ and tends to zero for $\ell_C\ge k+1$; that is, the largest possible order for the expression in \eqref{noAB} is achieved when $\ell_C=k+1$, when $r_{A\setminus B,C}+r_{A\cap B,C}=2$ and $r_{B\setminus A,C}=0$.  Using these values in \eqref{noAB} yields 
\begin{equation}\label{noAB2}
\frac{1}{\sigma^3}n^{k-1}p^{\binom{k+1}{2}-1}.
\end{equation}
Now suppose that only $\ell_A$ is fixed, and ignore the part of the summand depending only on its value.  We thus must consider summands of the size
\begin{equation}\label{noA}
\frac{1}{\sigma^3}n^{\ell_B-r_{A,B}+k-1}p^{\binom{\ell_B}{2}-\binom{r_{A,B}}{2}+\binom{k+1}{2}-1}.
\end{equation}
As before, if $\ell_B$ is increased and so is $r_{A,B}$, then the expression can only get smaller.  If $\ell_B$ is increased by 1 while $r_{A,B}$ stays fixed, then the ratio of the new expression to the old is $np^{\ell_B}$, and so we once again see that the largest possible size of the expression in \eqref{noA} is achieved when $\ell_B=k+1$ and $r_{A,B}=2$; the quantity in \eqref{noA} is thus bounded above by 
\begin{equation}\label{noA2}
\frac{1}{\sigma^3}n^{2k-2}p^{2\binom{k+1}{2}-2}.
\end{equation}
Finally, considering the full term, we have the upper bound of 
\begin{equation}\label{withA}
\frac{1}{\sigma^3}n^{\ell_A+2k-2}p^{\binom{\ell_A}{2}+2\binom{k+1}{2}-2};
\end{equation}
by the same argument one last time, this expression is maximized when $\ell_A=k+1$, yielding
\begin{equation}\label{finaltermbound}
\frac{1}{\sigma^3}n^{3k-1}p^{3\binom{k+1}{2}-2};
\end{equation}
that is, Theorem \ref{BKR-thm} implies that 
\[d_1(W,Z)\le\frac{C}{\sigma^3}n^{3k-1}p^{3\binom{k+1}{2}-2},\]
where $W=\frac{\tilde{\beta}_k-\E\tilde{\beta}_k}{\sqrt{\var(\tilde{\beta}_k)}}.$

The computation of $\sigma^2$ from \cite{KM} essentially goes through as before.  It was shown there that
\[\var(f_k)\sim c_kn^{2k}p^{2\binom{k+1}{2}-1}\]
(the numbered subclaim and equation (4) of \cite{KM} are inconsistent and in fact both wrong: unfortunate casualties of a change of index in the course of editing).
From this it follows that for any $j>0$,
\[\frac{\var(f_{k\pm j})}{\var(f_k)}\to 0;\]
moreover, one can compute covariances as in \cite{KM}, yielding for example the formula 
\[\cov(f_{k+j},f_{k+\ell})=\binom{n}{k+j+1}p^{\binom{k+j+1}{2}+\binom{k+\ell+1}{2}}\sum_{r=2}^{k+j+1}\binom{k+j+1}{r}\binom{n-k-j-1}{k+\ell+1-r}\left(p^{-\binom{r}{2}}-1\right),\]
for $0\le j\le k$ (and similarly in other cases).  Again one confirms that the order of this expression is smaller than the order of the variance of $f_k$, so that we finally have that 
\[\sigma^2\sim n^{2k}p^{2\binom{k+1}{2}-1}.\] 

The sums over indices only contribute constants depending on $k$, so that we have that the error in the abstract normal approximation theorem above is bounded above by
\[c_k\frac{n^{3k-1}p^{3\binom{k+1}{2}-2}}{n^{3k}p^{3\binom{k+1}{2}-\frac{3}{2}}}=\frac{c_k}{n\sqrt{p}},\]
for a constant $c_k$ depending only on $k$.

\section{The \v{C}ech complex}
In \cite{KM} we claimed three different limit theorems for the $k$th Betti number of a random \v{C}ech complex: depending on the sub-regime of the sparse regime, the $k$th Betti number either vanished a.a.s., had an approximate Poisson distribution, or satisfied a central limit theorem.
The approach taken in \cite{KM} works in most of the sparse regime, namely as long as $n^{k+3}r_n^{d(k+2)}\to 0$, but
to deal with the regime
in which $r_n=o(n^{-1/d - \delta})$ for some $\delta > 0$, but $n^{k+3}r_n^{d(k+2)}$ is bounded away from zero, a slightly different argument is needed, for the same reason as in the previous section.  

 We begin by noting that one can write $\beta_k$ semi-explicitly as follows. 
Let $S_k$ denote the number of empty k+1-dimensional simplex components of the \v{C}ech complex $\mathcal{C}(X_1,\ldots,X_n)$ spanned by $X_1,\ldots,X_n$.
Note that every such connected component has exactly $k+2$ vertices.

For every pair of integers $i > k+2$ and $j > 0$, let $X_{i,j}$ denote the number of
connected components $C$ of $\mathcal{C}(X_1,\ldots,X_n)$ on $i$ vertices such that $\beta_k(C) = j$. In other
words $X_{i,j}$ counts the components on $i$ vertices which contribute exactly $j$ to $\beta_k$.

Then 

\[\beta_k = S_k + \sum_{i > k+2, j > 0} j X_{i,j}.\]

A central limit theorem for $\beta_k$ is a indeed a consequence of a central limit theorem for $S_k$ as claimed in \cite{KM}, by a slightly more careful analysis.
%

Set $m = \lfloor 1 + 1 /( \delta d) \rfloor$, and define the truncated sum

\[\tilde{\beta}_k = S_k + \sum_{i = k+3 }^{m} \sum_{ j > 0} j X_{i,j}.\]

By a modification of the argument in \cite{KM}, one obtains the following.
\begin{thm}\label{clt-tilde}
With notation as above, for $r_n=o\left(n^{-1/d - \delta}\right)$ and $\lim_{n\to\infty}n^{k+2}r_n^{d(k+1)}\to\infty$,
$$\frac{1}{\sqrt{n(nr_n^d)^{k+1}}}\Big(\tilde{\beta}_{k}-\E\left[\tilde{\beta}_{k}\right]\Big)
\Longrightarrow\mathcal{N}\left(0,\frac{\mu_{k+2,1}}{(k+2)!}\right),$$
where $\mu_{k+2,1}$ is a constant depending only on $f$ and $k$.
\end{thm}

From here, a central limit theorem for $\beta_k$ itself follows:
\begin{thm}\label{final-clt}
With notation as above, for $r_n=o\left(n^{-1/d - \delta}\right)$ and $\lim_{n\to\infty}n^{k+2}r_n^{d(k+1)}\to\infty$,
$$\frac{1}{\sqrt{n(nr_n^d)^{k+1}}}\Big(\beta_{k}-\E\left[\beta_{k}\right]\Big)
\Longrightarrow\mathcal{N}\left(0,\frac{\mu_{k+2,1}}{(k+2)!}\right),$$
where $\mu_{k+2,1}$ is the same constant as in Theorem \ref{clt-tilde}.
\end{thm}
 Compare with Theorem 3.2 (iii) of \cite{KM}: the range of $r_n$ is slightly more restricted (there, $\delta$ was taken to be 0); the theorem here is also stated in terms of a specific numerical normalization, rather than abstractly in terms of the variance of $\beta_k$ as in \cite{KM}.

\begin{proof}[Proof of Theorem \ref{final-clt} from Theorem \ref{clt-tilde}]
Observe that 
\begin{equation*}\begin{split}
\P\left[\frac{\beta_k-\E\beta_k}{\sqrt{n(nr_n^d)^{k+1}}}\le t\right]&\le\P\left[\frac{\tilde{\beta}_k-\E\tilde{\beta}_k}{\sqrt{n(nr_n^d)^{k+1}}}\le t+\epsilon\right]+\P\left[\left|\frac{\tilde{\beta}_k-\beta_k-\E[\tilde{\beta}_k-\beta_k]}{\sqrt{n(nr_n^d)^{k+1}}}\right|>\epsilon\right]\\&\le\P\left[\frac{\tilde{\beta}_k-\E\tilde{\beta}_k}{\sqrt{n(nr_n^d)^{k+1}}}\le t+\epsilon\right]+\frac{2\E|\beta_k-\tilde{\beta}_k|}{\epsilon \sqrt{n(nr_n^d)^{k+1}}}.\end{split}\end{equation*}

\noindent {\bf Claim:} $\displaystyle\frac{\E|\beta_k-\tilde{\beta}_k|}{\sqrt{n(nr_n^d)^{k+1}}}\xrightarrow{n\to\infty}0$.

From the claim it follows that, given $\epsilon>0$, there is an $n$ large enough so that 
\[\P\left[\frac{\beta_k-\E\beta_k}{\sqrt{n(nr_n^d)^{k+1}}}\le t\right]\le\P\left[\frac{\tilde{\beta}_k-\E\tilde{\beta}_k}{\sqrt{n(nr_n^d)^{k+1}}}\le t+\epsilon\right]+\epsilon.\]
Using the central limit theorem already established for $\tilde{\beta}_k$ and then letting $\epsilon\to0$ shows that 
\[\limsup_{n\to\infty}\P\left[\frac{\beta_k-\E\beta_k}{\sqrt{n(nr_n^d)^{k+1}}}\le t\right]\le \P\left[\sqrt{\frac{\mu_{k+2,1}}{(k+2)!}}Z\le t\right],\]
where $Z$ is a standard Gaussian random variable.
The corresponding lower bound follows in the same way, so that given the claim, the proof of Theorem \ref{final-clt} is complete.


To prove the claim, observe that

\begin{equation}\begin{split}\label{beta-diff-upper-bound}
|\beta_k-\tilde{\beta}_k|=\sum_{i=m+1}^S\sum_{j>0}jX_{i,j}+\sum_{i=S+1}^n\sum_{j>0}X_{i,j},
\end{split}\end{equation}
where $S=\left\lceil\frac{2}{d\delta}+1\right\rceil$.

Since there are $i^{i-2}$ spanning trees on a set of $i$ vertices, and since a connected component of order $i$ can contribute at most $i \choose k+1$ to $\beta_k$, we have that for fixed $i\ge m+1$, 
\begin{align*}
\E \left[ \sum_{ j > 0} j X_{i,j}  \right] & \le  {n \choose i} i^{i-2} \left(\|f\|_\infty r_n^{d}\right)^{(i-1)} {i \choose k+1}  \le  \frac{n^i}{i!} i^{i-2} \left(\|f\|_\infty r_n^{d}\right)^{(i-1)} {i \choose k+1}.
\end{align*}
It follows that 
\begin{align*}
2S \sum_{i \ge m+1}^S  \E \left[ \sum_{ j > 0} j X_{i,j}  \right] &=O\left(n^{m+1}r_n^{dm}\right)= O \left( n^{-d \delta} \right),\\
\end{align*}
since $r_n = o(n^{-1/d - \delta})$; this takes care of the first sum in \eqref{beta-diff-upper-bound}.

For the second sum, the same estimate on the terms gives that 
\[\sum_{i=S+1}^n \E \left[  \sum_{ j > 0} j X_{i,j} \right]=O\left(n^2(nr_n^d)^S\right)=O\left(n^{2-d\delta S}\right).\]
Since $S>\frac{2}{d\delta}+1$, we have that $
\E \left|\beta_k - \tilde{\beta}_k \right| =O \left( n^{- \delta d} \right),$ which proves the claim.

\end{proof}

\bigskip

As in \cite{KM}, to prove Theorem \ref{clt-tilde}, we consider the Poissonized problem first, then recover the i.i.d.\ case.

Let $N_n$ be a Poisson random variable with mean $n$, and let $\p_n=
\{X_1,\ldots,X_{N_n}\},$ where $\{X_i\}_{i=1}^\infty$ is an i.i.d.\ sequence
of random points in $\R^d$ with density $f$.  Then $\p_n$ is a Poisson 
process with intensity $nf(\cdot)$, and one can define $S_k^P$ and 
$X_{i,j}^P$ for the random points $\p_n$ analogously to the earlier 
definitions.  In what follows, assume that $k\ge 2$; that is, 
the empty $k$-simplices are at least empty triangles.  Empty 
1-simplices
are simply pairs of vertices which are not connected, and
different arguments are needed in that case.

In order to compute expectations for the expressions which arise in
the Poissonized case, the following results are
useful.

\begin{thm}[See Theorem 1.6 of \cite{penrose}]\label{one}
Let $\lambda>0$  and let $\p_\lambda$ be a Poisson process with intensity 
$\lambda f(\cdot)$.  Let 
$j\in\N$, and suppose that $h(\Y,\S)$ is a bounded measurable
function on pairs $(\Y,\S)$ with $\S$ a finite subset of $\R^d$ and 
$\Y\subseteq\S$, such that $h(\Y,\S)=0$ unless $|\Y|=j$.  Then
$$\E\left[\sum_{\Y\subseteq\p_\lambda}h(\Y,\p_\lambda)\right]=\frac{\lambda^j}{j!}
\E h(\X_j,\X_j\cup\p_\lambda),$$
where $\X_j$ is a set of $j$ i.i.d. points in $\R^d$ with density $f$, 
independent of $\p_\lambda$.
\end{thm}

From this, one can prove the following (see \cite{KM} for the proof).
\begin{thm}\label{product}
Let $\lambda>0$ and $k,j_1,\ldots,j_k\in\N$; define $j:=\sum_{i=1}^kj_i$.  
For $1\le i\le k$, suppose
$h_i(\Y,\S)$ is a bounded measurable function of pairs $(\Y,\S)$ of
finite subsets of $\R^d$ with $\Y\subseteq\S$, such that $h_i(\Y,\S)=0$
if $|\Y|\neq j_i$.  Then
$$\E\left[\sum_{\Y_1,\subseteq\p_\lambda}\cdots\sum_{\Y_k\subseteq\p_\lambda}\left(
\prod_{i=1}^kh_i(\Y_i)\right)\1_{\{\Y_i\cap\Y_j=\emptyset\,{\rm for }\,
i\neq j\}}\right]=\E\left[\prod_{i=1}^k\left(\frac{\lambda^{j_i}}{j_i!}\right)
h_i(\X_{j_i},\left(\cup_{i=1}^{k}\X_{j_i}\right)\cup\p_n)\right],$$
\end{thm}
where $\X_{j_i}$ are $j_i$ i.i.d points in $\R^d$ with density $f$,  $\p_\lambda$
is a Poisson process with intensity $\lambda f(\cdot)$, and $\{\X_{j_i}\}_{i=1}^k$ and 
$\p_\lambda$ are  all independent.

\medskip

One can apply these results to compute the mean and variance of $\tilde{\beta}_{k,A}^P$, the contribution to $\tilde{\beta}_k^P$ from components 
whose left-most vertex is in an open set $A$ with $\vol(\partial A)=0$.  

In order to apply the lemmas, the corresponding means in the i.i.d.\ case are needed.

\begin{lemma}\label{exp-order}
Let $g_{i,j,A}(x_1,\ldots,x_i)$ be the indicator that $\mathcal{C}(\{x_1,\ldots,x_i\})$ is connected, has $k$th Betti number equal to $j$, and has left-most-point in $A$.  Then for $\{X_i\}$ i.i.d. with density $f$ as before, there is a constant $\mu_{i,j,A}$ depending only on $i,j, f,$ and $A$ such that
\[\lim_{n\to\infty}r_n^{-d(i-1)}\E\left[g_{i,j,A}(X_1,\ldots,X_i)\right]=\mu_{i,j,A}.\]
\end{lemma}

The proof is identical to that of the analagous result in Chapter 3 of Penrose \cite{penrose}.

\begin{lemma}\label{CC_Poisson_means}
For $\mu_{i,j,A}$ as in Lemma \ref{exp-order},
\begin{enumerate}
\item $$\lim_{n\to\infty}n^{-(k+2)}r_n^{-d(k+1)}\E\left[\tilde{\beta}_k^P\right]=
\lim_{n\to\infty}n^{-(k+2)}r_n^{-d(k+1)}\var\left[\tilde{\beta}_{k}^P\right]=\frac{\mu_{k+2,1,A}}{(k+2)!}.$$
\item $$\lim_{n\to\infty}n^{-(k+2)}r_n^{-d(k+1)}\E\left[S_k^P\right]=
\lim_{n\to\infty}n^{-(k+2)}r_n^{-d(k+1)}\var\left[S_{k}^P\right]=\frac{\mu_{k+2,1,A}}{(k+2)!}.$$
\item \begin{equation*}\begin{split}
\lim_{n\to\infty}&n^{-(k+3)}r_n^{-d(k+2)}\E\left[\sum_{\substack{k+3\le i\le m\\j\ge 0}}jX_{i,j}\right]\\&=\lim_{n\to\infty}n^{-(k+3)}r_n^{-d(k+2)}\var\left[\sum_{\substack{ k+3\le i\le m\\j\ge 0}}jX_{i,j}\right]=\frac{1}{(k+3)!}\sum_{j=1}^{\binom{k+3}{k+1}}j\mu_{k+3,j,A}.\end{split}\end{equation*}
\end{enumerate}
\end{lemma}

\begin{proof}
For $i\ge k+2$ and $j\ge 1$, let $h_{i,j,A}(\{x_0,\ldots, x_k\},\X)$ be the indicator that
$\{x_0,\ldots,x_{i-1}\}\subseteq\X$ form a connected component of $\mathcal{C}(\X)$ with $\tilde{\beta}_k(\mathcal{C}(x_0,\ldots,x_{i-1}))=j$, whose left-most point is in $A$.  Then with $\tilde{\beta}_{k,A}^P$ denoting the sub-sum of $\tilde{\beta}_k$ coming from those components with left-most point in $A$,
\begin{equation}\begin{split}\label{mean-unPoisson}
\E[\tilde{\beta}_{k,A}^P]&=\E\left[\sum_{\substack{ k+2\le i\le m\\j\ge 1}}\sum_{\Y\subseteq\p_\lambda}jh_{i,j,A}(\Y,\p_n)\right]
=\sum_{\substack{k+2\le i\le m\\j\ge 1}}\frac{n^{i}}{i!}\E\left[jh_{i,j,A}(\X_{i},\X_{i}\cup\p_n)\right].
\end{split}\end{equation}
Now, $\E\left[h_{i,j,A}(\X_{i},\X_{i}\cup\p_n)\right]\le\E\left[g_{i,j,A}(\X_{i})
\right]$, where $g_{i,j,A}(\X_{i})$ is the indicator that the $i$ i.i.d. points $\X_i$ are connected (with respect to cut-off radius $r_n$) with $k$th Betti number of the complex they span equal to $j$ (ignoring any issues of connectedness to anything else).  By Lemma \ref{exp-order} $\E\left[g_{i,j,A}(\X_{i})
\right]\simeq r_n^{d(i-1)}\mu_{i,j,A}$.  Note moreover that the conditional probability that
$\X_{i}$ is isolated from $\p_n$ given that $\X_{i}$ is connected and has left-most vertex in $A$
is bounded below by the probability that there are no points of $\p_n$ in
the ball of radius $2(\ell_{i,j}+1)r_n$ about $X_1$, where $\ell_{i,j}$ is the largest number of edges that may be needed to move from one vertex to another in a simplicial complex on $i$ vertices with $k$th Betti number equal to $j$.  Since $\p_n$ is a Poisson 
process with intensity $nf(\cdot)$, this probability is given exactly by $e^{-n\vol_f(B_{2(\ell_{i,j}+1)r_n}(X_1)
)}\ge e^{-n\|f\|_\infty\theta_d(2(\ell_{i,j}+1)r_n)^d}.$  It thus follows that 
\begin{equation*}\begin{split}
 \E\left[g_{i,j,A}(\X_{i},\X_{i}\cup\p_n)\right]&\ge e^{-n\|f\|_\infty\theta_d(2(\ell_{i,j}+1)r_n)^d}
\E[g_{i,j,A}(\X_{i})]\\&\simeq e^{-n\|f\|_\infty\theta_d(2(\ell_{i,j}+1)r_n)^d}r_n^{d(i-1)}\mu_{i,j,A}.
\end{split}\end{equation*}
Recall that for  $i\ge k+2$ fixed, $j\le\binom{i}{k+1}$.  It thus follows that since $nr_n^d\to0$,

\begin{equation*}
\E\left[\sum_{j>0}jX_{i,j,A}^P\right]\simeq \frac{n^ir_n^{d(i-1)}}{i!}\sum_{j=1}^{\binom{i}{k+1}}j\mu_{i,j,A}, 
\end{equation*}
and that in particular, 
\[\E\tilde{\beta}_{k,A}^P\simeq \frac{n^{k+2}r_n^{d(k+1)}}{(k+2)!}\mu_{k+2,1,A}.\]

 A similar approach is taken to compute the variance:
\begin{equation*}\begin{split}
\E\left[(\tilde{\beta}_{k,A}^P)^2\right]&
=\E\left[\sum_{\Y\subseteq\p_n}\sum_{i,i'=k+2}^m\sum_{j,j'>0}jj' h_{i,j,A}(\Y,\p_n) h_{i',j',A}(\Y,\p_n)
\right]\\&\qquad\qquad
+\E\left[\sum_{\ell=0}^{m}\sum_{\substack{\Y,\Y'\subseteq\p_n\\\Y\neq \Y'}}\sum_{i,i'=k+2}^m\sum_{j,j'>0}jj'h_{i,j,A}(\Y,\p_n)
h_{i',j',A}(\Y',\p_n)\1_{\{|\Y\cap\Y'|=\ell\}}\right].
\end{split}\end{equation*}
For the first term, note that $h_{i,j,A}(\Y,\p_n) h_{i',j',A}(\Y,\p_n)=0$ unless $i=i'$ and $j=j'$, because $i$ is the number of vertices of $\Y$ and $j$ is the $k$th Betti number of the complex it spans.  This means the first term has in fact already been analyzed:
\begin{equation*}\begin{split}\E&\left[\sum_{\Y\subseteq\p_n}\sum_{i,i'=k+2}^m\sum_{j,j'>0}jj h_{i,j,A}(\Y,\p_n) h_{i',j',A}(\Y,\p_n)
\right]\\&\qquad=\E\left[\sum_{\Y\subseteq\p_n}\sum_{i=k+2}^m\sum_{j>0}j^2 h_{i,j,A}(\Y,\p_n)
\right]\simeq \frac{n^{k+2}r_n^{d(k+1)}}{(k+2)!}\mu_{k+2,1,A}.\end{split}\end{equation*}

For the second, observe first that the terms corresponding to $\ell\neq 0$ 
vanish:\\ $h_{i,j,A}(\Y,\p_n)h_{i',j',A}(\Y',\p_n)\equiv 0$ if $|\Y\cap\Y'|=\ell>0$,
because in that case neither $\Y$ nor $\Y'$ is a whole component.  
When $\ell=0$, applying Theorem \ref{product} yields
\begin{equation*}\begin{split}
\E&\left[\sum_{\Y,\Y'\subseteq\p_n}h_{i,j,A}(\Y,\p_n)h_{i',j',A}
(\Y',\p_n)\1_{\{\Y\cap\Y'=\emptyset\}}
\right]\\&\qquad\qquad
=\frac{n^{i+i'}}{i!i'!}\E\Big[jj'h_{i,j,A}(\X_i,\X_i\cup\X_{i'}\cup\p_n)
h_{i',j',A}(\X_{i'},\X_i\cup\X_{i'}\cup\p_n)\Big],
\end{split}\end{equation*} 
where again $\X_i$ and $\X_{i'}$ are independent collections of $i$ and $i'$ i.i.d. points distributed according to $f$, respectively.
Making use of \eqref{mean-unPoisson} thus yields
\begin{equation*}\begin{split}
\var&\left[\tilde{\beta}_{k,A}^P\right]\\&=\E\left[\tilde{\beta}_{k,A}^P\right]+\sum_{i,i'= k+2}^m\sum_{j,j'\ge 1}
\frac{n^{i+i'}jj'}{i!i'!}
\Bigg\{\E\left[h_{i,j,A}(\X_{i},\X_i\cup\X_{i'}\cup\p_n)
h_{i',j',A}(\X_{i'},\X_i\cup\X_{i'}\cup\p_n)\right]\\&\qquad\qquad\qquad\qquad\qquad\qquad\qquad\qquad\qquad\qquad-\E\left[
h_{i,j,A}(\X_{i},\X_i\cup\p_n)\right]\E\left[
h_{i',j',A}(\X_{i'},\X_{i'}\cup\p_n)\right]\Bigg\}.
\end{split}\end{equation*}
Now, let $\p_n'$ be an independent copy of $\p_n$. Then
\begin{equation*}\begin{split}
\E&\left[h_{i,j,A}(\X_i,\X_i\cup\X_{i'}\cup\p_n)
h_{i',j',A}(\X_{i'},\X_i\cup\X_{i'}\cup\p_n)\right]\\&\qquad-\E\left[
h_{i,j,A}(\X_{i},\X_i\cup\p_n)\right]\E\left[
h_{i',j',A}(\X_{i'},\X_{i'}\cup\p_n)\right]\\&=
\E\Bigg[h_{i,j,A}(\X_i,\X_i\cup\X_{i'}\cup\p_n)
h_{i',j',A}(\X_{i'},\X_i\cup\X_{i'}\cup\p_n)\\&\qquad\qquad-h_{i,j,A}(\X_{i},\X_i\cup\p_n) h_{i',j',A}(\X_{i'},\X_{i'}\cup\p_n')\Bigg]\\&=
\E\left[\left(h_{i,j,A}(\X_i,\X_i\cup\X_{i'}\cup\p_n)-h_{i,j,A}(\X_{i},\X_i\cup\p_n)\right) h_{i',j',A}(\X_{i'},\X_i\cup\X_{i'}\cup\p_n)\right]\\&\quad+
\E\left[h_{i,j,A}(\X_i,\X_i\cup\p_n)
\left(h_{i',j',A}(\X_{i'},\X_i\cup\X_{i'}\cup\p_n)-h_{i',j',A}(\X_{i'},\X_{i'}\cup\p_n)
\right)\right]\\&\quad+\E\left[h_{i,j,A}(\X_i,\X_i\cup\p_n)
\left(h_{i',j',A}(\X_{i'},\X_{i'}\cup\p_n)- h_{i',j',A}(\X_{i'},\X_{i'}\cup\p_n'
\right)\right]\\&=:E_1+E_2+E_3.
\end{split}\end{equation*}
Now, observe that in fact $E_1=0$: the difference is non-zero if and
only if $\X_i$ and $\X_{i'}$ are connected by an edge, in
which case the second factor is zero.

Observe that the difference in $E_2$ is either 0 or $-1$ .  Furthermore, 
it is non-zero if and only if $\X_{i}$ and $\X_{i'}$ are
connected by an edge, and both $\X_i$ and $\X_{i'}$
are connected.  This probability is bounded above by
\[\|f\|_\infty^{i+i'-1}\theta_d^{i+i'-1}(2\ell_{i,j}r_n)^{d(i-1)}(2\ell_{i',j'}r_n)^{d(i'-1)} (2(\ell_{i,j}+\ell_{i'+j'}+1)r_n)^d.\]

Finally, conditional on the event $\left[\cup_{x\in\X_i}B_{2\ell_{i,j}r_n}(x)\right]\cap 
\left[\cup_{x\in\X_{i'}}B_{2\ell_{i',j'}r_n}(x)\right]=\emptyset$, the 
two terms of $E_3$ have the same distribution by the spacial
independence property of the Poisson process.  A contribution from
$E_3$ therefore only arises if in particular $\X_i$ and $\X_{i'}$ are both connected and the intersection above is non-empty. The probability of this event is bounded
above by 
\[\|f\|_\infty^{i+i'-1}\theta_d^{i+i'-1}(2\ell_{i,j}r_n)^{d(i-1)}(2\ell_{i',j'}r_n)^{d(i'-1)} (4(\ell_{i,j}+\ell_{i',j'})r_n)^d.\]
It follows that
\[\var\left[\tilde{\beta}_{k,A}^P\right]= \E\left[\tilde{\beta}_{k,A}^P\right]+E,\]
and
\[
|E|\le\sum_{i,i'= k+2}^m\sum_{j,j'\ge 1}\frac{n^{i+i'}jj'(C_{i,j}r_n)^{d(i+i'-1)}}{i!i'!}2\|f\|_\infty^{i+i'-1}\theta_d^{i+i'-1}\le C(f,d,k)
(nr_n^d)^{k+2}(n^{k+2}r_n^{d(k+1)}),\]
where $C(f,k,d)$ is a constant depending on $f$, $d$, and $k$.  This
completes the proof of the first statement of the lemma.  The proof of the second statement is the same, just removing the terms of the sum indexed by $i>k+2$, and the third statement is gotten by removing the terms indexed by $i=k+2$. 
\end{proof}

The following was proved via Stein's method in \cite{KM}.
\begin{thm}\label{Poissonized_normal}
With notation as above, and for $n^{k+2}r_n^{d(k+1)}\to\infty$ and $nr_n^d\to0$,
$$\frac{S_k^P-\E\left[S_k^P\right]}{\sqrt{n^{k+2}r_n^{d(k+1)}}}
\Longrightarrow\mathcal{N}\left(0,\frac{\mu_{k+2,1}}{(k+2)!}\right).$$
\end{thm}

This gives a central limit theorem for $\tilde{\beta}_k^P$ as follows.  Write 
\[\tilde{\beta}_k^P=S_k^P+R_k^P,\qquad R_k^P:=\sum_{\substack{i\ge k+3\\j\ge 1}}jX_{i,j}.\]
Fix $t\in\R$ and $\epsilon>0$.  By Lemma \ref{CC_Poisson_means}, for $n$ large enough, $\frac{\var(R_k^P)}{n^{k+2}r_n^{d(k+1)}}\le\epsilon^3$, so that for $n$ large enough,   
\begin{equation*}\begin{split}
\P\left[\frac{\tilde{\beta}_k^P-\E\tilde{\beta}_k^P}{\sqrt{n^{k+2}r_n^{d(k+1)}}}\le t\right]&\le\P\left[\frac{S_k^P-\E S_k^P}{\sqrt{n^{k+2}r_n^{d(k+1)}}}\le t+\epsilon\right]+\P\left[\left|\frac{R_k^P-\E R_k^P}{\sqrt{n^{k+2}r_n^{d(k+1)}}}\right|>\epsilon\right]\\
&\le \P\left[\frac{S_k^P-\E S_k^P}{\sqrt{n^{k+2}r_n^{d(k+1)}}}\le t+\epsilon\right]+\epsilon.
\end{split}\end{equation*}
By applying the central limit theorem already proved for $S_k^P$ and then letting $\epsilon\to0$, it follows that 
\[\limsup_{n\to\infty}\P\left[\frac{\tilde{\beta}_k^P-\E\tilde{\beta}_k^P}{\sqrt{n^{k+2}r_n^{d(k+1)}}}\le t\right]\le\P\left[\sqrt{\frac{\mu_{k+2,1}}{(k+2)!}}Z\le t\right].\]
The other inequality is proved in the same way, giving the following central limit theorem for $\tilde{\beta}_k^P$.

\begin{thm}\label{Poissonized_CLT}
For notation as above, 
\[\frac{1}{\sqrt{n^{k+2}r_n^{d(k+1)}}}\left(\tilde{\beta}_k^P-\E\tilde{\beta}_k^P\right)\Longrightarrow \mathcal{N}\left(0,\frac{\mu_{k+2,1}}{(k+2)!}\right).\]
\end{thm}

\medskip

The remaining work is to use this result to obtain the same 
result for $\tilde{\beta}_{k}$ itself. 
To do so, the following ``de-Poissonization result'' is used.
\begin{thm}[Theorem 2.12 of \cite{penrose}]\label{de-Poisson}
Suppose that for each $n\in\N$, $H_n(\X)$ is a real-valued functional on finite
sets $\X\subseteq\R^d$.  Suppose that for some $\sigma^2\ge 0$, 
\begin{enumerate}
\item $\ds\frac{1}{n}\var(H_n(\p_n))\longrightarrow\sigma^2,$
and
\item $\ds\frac{1}{\sqrt{n}}\big[H_n(\p_n)-\E H_n(\p_n)\big]
\Longrightarrow\sigma^2Z,$
for $Z$ a standard normal random variable.
\end{enumerate}
Suppose that there are constants $\alpha\in\R$ and $\gamma>\frac{1}{2}$ such
that the increments $R_{q,n}=H_n(\X_{q+1})-H_n(\X_q)$ satisfy
\begin{equation}\label{means}
\lim_{n\to\infty}\left(\sup_{n-n^\gamma\le q\le n+n^\gamma}|\E[R_{q,n}]-\alpha|
\right)=0,
\end{equation}
\begin{equation}\label{covs}
\lim_{n\to\infty}\left(\sup_{n-n^\gamma\le q<q'\le n+n^\gamma}|\E[R_{q,n}R_{q',n}]
-\alpha^2|\right)=0,
\end{equation}
and
\begin{equation}\label{vars}
\lim_{n\to\infty}\left(\frac{1}{\sqrt{n}}\sup_{n-n^\gamma\le q\le n+n^\gamma}
\E[R_{q,n}^2]\right)=0.
\end{equation}
Finally, assume that there is a constant $\tilde{\beta}>0$ such that, with probability
one, 
$$|H_n(\X_q)|\le\tilde{\beta}(n+q)^\beta.$$
Then $\alpha^2\le\sigma^2$ and as $n\to\infty$, $\frac{1}{n}\var(H_n(\X_n))
\to \sigma^2-\alpha^2$ and 
$$\frac{1}{\sqrt{n}}\big[H_n(\X_n)-\E H_n(\X_n)\big]\Longrightarrow
\sqrt{\sigma^2-\alpha^2}Z.$$

\end{thm}

In conjunction with Theorem \ref{Poissonized_normal}, this yields Theorem \ref{clt-tilde}, as follows.

\begin{proof}[Proof of Theorem \ref{clt-tilde}]
Theorem \ref{de-Poisson} is applied to the functional  
$$H_n(\X):=\frac{1}{\sqrt{(nr_n^d)^{k+1}}}\sum_{\Y\subseteq\X}\left(\sum_{i= k+2}^m\sum_{j\ge 1}jh_{i,j}
(\Y,\X)\right);$$
$\sigma^2=\frac{\mu_{k+2,1}}{(k+2)!}$ and the central limit theorem
holds for $H_n(\p_n)$ by Theorem \ref{Poissonized_CLT}.  

Let 
\[D_{q,n}:=\sum_{\Y\subseteq\X_{q+1}}\left(\sum_{i= k+2}^m\sum_{j\ge 1}jh_{i,j}(\Y,\X_{q+1})\right)-\sum_{\Y\subseteq\X_q}\left(\sum_{i= k+2}^m\sum_{j\ge 1}j
h_{i,j}(\Y,\X_q)\right)\] 
(the dependence on $n$ is only through the threshhold radius $r_n$), and observe that $D_{q,n}$
 is the $k$th Betti number of the component of $X_{q+1}$ in 
$\X_{q+1}$, minus the $k$th Betti number of the complex that results by taking the component of $X_{q+1}$ and removing $X_{q+1}$ from it, assuming these components are on $m$ or fewer vertices.  
It follows that the difference is bounded by $\binom{m}{k+1}$, and is only non-zero if $X_{q+1}$ is connected to at least $k+1$ other vertices, so that
\[\left|\E[D_{q,n}]\right|\le \binom{m}{k+1}\binom{n+n^\gamma}{k+1}\left(\|f\|_\infty r_n^{d}\right)^{(k+1)}\le \binom{m}{k+1}\left((n+n^\gamma) \|f\|_\infty r_n^{d}\right)^{(k+1)}.\]
The first condition of the theorem is then satisfied with $\alpha=0$, for any $\gamma\in\left(\frac{1}{2},1\right]$.

Next, consider the quantity $\E[D_{q,n}D_{q',n}]$ for $q< q'$.
By the observation above, $D_{q,n}D_{q',n}$ is uniformly bounded by $\binom{m}{k+1}^2$.  The probability that $D_{q,n}\neq 0$ is bounded above by $c[n+n^\gamma]^{k+1}r_n^{d(k+1)}$ as before.  Given that that difference is non-zero, the largest probability event that causes $D_{q',n}$ to be non-zero is that $X_{q'+1}$ is connected to the component of $X_{q+1}$, and that its removal changes the Betti number of that component.  The probability that $X_{q'+1}$ is in the component of $X_{q+1}$ is bounded above $cr_n^d$, so that 
\[|\E[R_{q,n}R_{q',n}]|=\frac{1}{(nr_n^d)^{k+1}}|\E[D_{q,n}D_{q',n}]|\le\frac{1}{(nr_n^d)^{k+1}}\left(c[n+n^\gamma]^{k+1}r_n^{d(k+2)}\right)\le cr_n^d,\]
so that the second condition of the theorem is also satisfied.  

If $q=q'$, then we have as above
\[\E[D_{q,n}^2]\le c[n+n^\gamma]^{k+1}r_n^{d(k+1)},\]
so that 
\[\frac{1}{\sqrt{n}}\E[R_{q,n}^2]\le\frac{2^{k+1}c}{\sqrt{n}},\]
and so the third condition is satisfied as well.

Finally, the polynomial boundedness condition of Theorem \ref{de-Poisson}
is satisfied trivially, and the proof is complete.

\end{proof}

\bibliographystyle{plain}
\bibliography{empty}

\end{document}